\documentclass[11pt]{amsart}
\usepackage {enumerate}
\usepackage{setspace}
\usepackage{caption}
\usepackage{mathrsfs}
\usepackage{hyperref}
\usepackage{esint}
\usepackage{amssymb}

\usepackage{color}
\usepackage{lipsum}
\newcommand\blfootnote[1]{%
  \begingroup
  \renewcommand\thefootnote{}\footnote{#1}%
  \addtocounter{footnote}{-1}%
  \endgroup
}
\newtheorem{theorem}{Theorem}[section]
\newtheorem{lemma}[theorem]{Lemma}

\newtheorem{proposition}[theorem]{Proposition}

\numberwithin{equation}{section}

\theoremstyle {definition}
\newtheorem{definition}[theorem]{Definition}
\newtheorem{remark}[theorem]{Remark}

\DeclareMathOperator{\area}{area}

\DeclareMathOperator{\vol}{vol}
\DeclareMathOperator{\Ric}{Ric}

\DeclareMathOperator{\tr}{tr}
\DeclareMathOperator{\Div}{div}

\DeclareMathOperator{\lip}{Lip}

\DeclareMathOperator{\sys}{sys}
\DeclareMathOperator{\id}{id}
\begin{document}
\title[Rigidity results for complete manifolds]{Rigidity results for complete manifolds with nonnegative scalar curvature}
\author{Jintian Zhu}
\address{Key Laboratory of Pure and Applied Mathematics, School of Mathematical Sciences, Peking University, Beijing, 100871, P.~R.~China}
\email{zhujt@pku.edu.cn, jintian@uchicago.edu}
\maketitle

\begin{abstract}
In this paper, we are going to show some rigidity results for complete open Riemannian manifolds with nonnegative scalar curvature. Without using the famous Cheeger-Gromoll splitting theorem we give a new proof to a rigidity result for complete manifolds with nonnegative scalar curvature admitting a proper smooth map to $T^{n-1}\times \mathbf R$ with nonzero degree. Here we introduce a trick to obtain the compactness of limit hypersurface from locally graphical convergence. Based on the same idea we also establish an optimal $2$-systole inequality for several classes of complete Riemannian manifolds with positive scalar curvature and further prove a rigidity result for the equality case.
\end{abstract}

\section{Introduction}
During the past a few decades, many rigidity results for closed Riemannian manifolds with nonnegative scalar curvature have been established. It is well-known that every smooth metric on $T^3$ with nonnegative scalar curvature is flat, which was proved by Schoen and Yau in \cite{SY1979(2)} using the minimal surface method. Later, they generalized the same result to $T^n$ in \cite{SY1979} with a dimension reduction argument when the dimension $n$ is between $3$ and $7$. In their work \cite{GL1980}, Gromov and Lawson established the same result for $T^n$ in all dimensions with a totally different method based on Dirac operator.

For closed Riemannian manifolds with positive scalar curvature, there is also an interesting rigidity phenomenon involving both the scalar curvature and the $2$-systole. For a complete Riemannian manifold $(M,g)$ with nontrivial second homotopy group, Bray, Brendle and Neves introduced the following {\it spherical $2$-systole}
\begin{equation*}
 \sys_2( M, g):=\inf\left\{\area(\mathbf S^2,i^* g)\left|\begin{array}{c}
 \text{$i:\mathbf S^2\to M$ smooth such that}\\
 \text{$[i]$ homotopically nontrivial}
 \end{array}\right.\right\}.
\end{equation*}
With this notion they proved in \cite{BBN2010} that the spherical $2$-systole of every closed Riemannian $3$-manifold with nontrivial second homotopy group is no greater than $4\pi$ if its scalar curvature is no less than $2$. Furthermore, the equality holds if and only if the universal covering splits into the standard Riemannian product $\mathbf S^2\times \mathbf R$. Generally, this estimate cannot be true for higher dimensions only with the assumption on positivity of the scalar curvature. After imposing some further topological restriction, the author made a generalization in \cite{Z2020} for closed Riemannian manifolds with dimension between $3$ and $7$. In fact, he proved that for every closed Riemannian $n$-manifold with scalar curvature no less than $2$ admitting a nonzero degree smooth map to $\mathbf S^2\times T^{n-2}$, the spherical $2$-systole is no greater than $4\pi$. Similarly, the equality holds if and only if the universal covering splits into the standard Riemannian product $\mathbf S^2\times \mathbf R^{n-2}$.

Apart from these results for closed manifolds, there are various discussions for complete open Riemannian manifolds as well. For instance, Schoen and Yau proved the following result in \cite{SY1982}: if the fundamental group of an open $3$-manifold has a subgroup isometric to that of a closed surface with positive genus, then it cannot admit any complete metric with positive scalar curvature. Based on Dirac operator method, Gromov and Lawson made a relatively systematic discussion on the non-existence of complete metrics with positive scalar curvature for several classes of manifolds in \cite{GL1983}. For a complete metric with nonnegative scalar curvature but non-flat Ricci curvature, it is a standard trick to make a deformation for this metric along the Ricci tensor and then conformal it to a new one with positive scalar curvature. From this point of view, for those manifolds in \cite{SY1982} and \cite{GL1983} every complete metric with nonnegative scalar curvature must be Ricci flat. In most cases, the famous Cheeger-Gromoll splitting theorem in \cite{CG71} can be then applied to obtain certain rigidity results. In particular, we know that every complete metric on $T^{n-1}\times \mathbf R$ with nonnegative scalar curvature is flat from the work in \cite{SY1982} and \cite{GL1983}. However, very few knowledge is known for the spherical $2$-systole of complete Riemannian manifolds with positive scalar curvature compared to the closed case.

In this article, we will continue a discussion on rigidity results for complete Riemannian manifolds with nonnegative scalar curvature. Our first purpose is to give an alternative proof for the following result without the use of the Cheeger-Gromoll splitting theorem.
\begin{theorem}\label{Thm: main1}
For $3\leq n\leq 7$, let $(M^n,g)$ be an orientable complete open Riemannian manifold with nonnegative scalar curvature, which admits a smooth proper map $f:M\to T^{n-1}\times \mathbf R$ with nonzero degree. Then $(M,g)$ is isometric to the Riemannian product of a flat $(n-1)$-torus and the real line.
\end{theorem}

Let us sketch the idea for our proof. Just as the closed case, we tend to find an area-minimizing $T^{n-1}$ in this complete Riemannian manifold, after that the desired consequence follows from the standard variation argument and foliation argument. However, an obvious difficulty here is the lack of compactness to minimize the area functional in a chosen homology class. For a minimizing sequence it is possible that the area loses entirely at infinity, which is totally different from the closed case. To overcome this difficulty we follow the idea from the author's previous work in \cite{Z2020width}, where an optimal width estimate was obtained for open manifolds with positive sectional curvatures. The idea is that we can construct suitable $h$-minimizing boundaries in a chosen homology class as good approximations for area-minimizing hypersurfaces. The benefit of doing so is that all these $h$-minimizing boundaries will intersect with a fixed compact subset after a careful choice for function $h$. This is the key to make a limiting procedure possible. However, there is still a concern about the loss of topology at infinity and the challenge here is to obtain the compactness of the limit hypersurface. Different from the situation in \cite{Z2020width}, the compactness comes from a more delicate analysis on the limit hypersurface. The observation now is that the limit hypersurface has to be Ricci-flat since it is the limit of a sequence of $T^{n-1}$ with almost nonnegative scalar curvature and it has finite volume from our construction.
With the help of Calabi-Yau linear volume growth estimate for complete open Riemannian manifolds with nonnegative Ricci curvature, the limit hypersurface turns out to be compact.

We believe that the new idea may lead to potential application to other geometric problems. Along with the same idea, we generalize the optimal $2$-systole estimate to complete open Riemannian manifolds with positive scalar curvature.
In three dimension, we prove the following analogy to the Bray-Brendle-Neves $2$-systole estimate.

\begin{theorem}\label{Thm: main2}
Let $(M^3,g)$ be a complete open Riemannian $3$-manifold with nontrivial second homotopy group and positive scalar curvature. Then
\begin{equation}\label{Eq: main2}
\inf_M R(g)\cdot\sys_2(M,g)\leq 8\pi.
\end{equation}
The equality holds if and only if the universal covering of $(M,g)$ is isometric to the standard Riemannian product $\mathbf S^2\times \mathbf R$ up to scaling.
\end{theorem}

For higher dimensions we establish the following theorem corresponding to the author's previous work in \cite{Z2020}.

\begin{theorem}\label{Thm: main3}
For $3\leq n\leq 7$, let $(M^n,g)$ be an orientable complete open Riemannian manifold with positive scalar curvature, which admits a smooth proper map $f:M\to \mathbf S^2\times T^{n-3}\times \mathbf R$ with nonzero degree. Then the second homotopy group is non-trivial and
\begin{equation}\label{Eq: main3}
\inf_M R(g)\cdot\sys_2(M,g)\leq 8\pi.
\end{equation}
The equality holds if and only if the universal covering of $(M,g)$ is isometric to the standard Riemannian product $\mathbf S^2\times \mathbf R^{n-2}$ up to scaling.
\end{theorem}

%Similar to the situation for Theorem \ref{Thm: main1}, based on the Cheeger-Gromoll splitting theorem we are able to provide another proof for Theorem \ref{Thm: main2} and Theorem \ref{Thm: main3}. For the completeness of our discussion we include it as the last part of our paper.

This article will be organized as follows. In section 2, we will construct desired approximation $h$-minimizing boundaries for our later proof. In section 3, we will show Ricci curvature estimates for some manifolds those are limits of closed manifolds with certain topological restriction and curvature conditions. This plays a crucial role in our proof to obtain the compactness of the limit from our constructed $h$-minimizing boundaries. In section 4, we give a detailed proof for our main theorems.

\section*{Acknowledgement}
This work is partially supported by China Scholarship Council and the NSFC grants No. 11671015 and 11731001. The author would like to thank Professor Andr\'e Neves for many helpful conversations. He is also grateful to Professor Yuguang Shi for constant encouragements.

\section{The approximation $h$-minimizing boundary}\label{Sec: 2}
In this section, we will construct desired approximation $h$-minimizing boundaries for our later proof. Assume that $(M^n,g)$ is an orientable complete open Riemannian manifold and $\Sigma\subset M$ is an orientable closed hypersurface associated with a surjective signed distance function.

First we modify the signed distance function to a smooth one.
\begin{lemma}\label{Lem: function phi}
There is a proper and surjective smooth function $\phi:M\to\mathbf R$ such that $\phi^{-1}(0)=\Sigma$ and $\lip\phi<1$.
\end{lemma}
\begin{proof}
Since $\rho:M\to\mathbf R$ is the signed distance function to the given closed hypersurface $\Sigma$, it is a proper Lipschitz function with $\lip\rho=1$. Fix a small positive constant $\epsilon<1$ such that $\rho$ is smooth in the region $\{-2\epsilon< \rho<2\epsilon\}$ and define a smooth cut-off function $\eta$ such that $0\leq \eta\leq 1$, $\eta\equiv 1$ in $(-\epsilon,\epsilon)$ and $\eta\equiv 0$ outside $(-2\epsilon,2\epsilon)$.
Let $\phi_1=(1-\eta\circ\rho)\rho$, then $\phi_1$ vanishes round $\Sigma$ and $\lip\phi_1\leq C$ for some universal constant $C$. Through standard mollification we can obtain a smooth function $\phi_2$ vanishing around $\Sigma$ with $\lip\phi_2\leq 2C$. Moreover, $\phi_2$ has a fixed sign on either side of $\Sigma$. Define
$$
\phi=\frac{1}{4C+4}\left(\phi_2+\int_0^\rho\eta(\tau)\,\mathrm d\tau\right).
$$
It is easy to verify that $\phi$ satisfies all our requirements.
\end{proof}

Fix such a function $\phi$ and denote
$$\Omega_0=\{x\in M:\phi(x)<0\}.$$
Given any smooth function $h:(-T,T)\to \mathbf R$, we introduce the following functional
\begin{equation*}
\mathcal A^h(\Omega)=\mathcal H^{n-1}(\partial^*\Omega)-\int_{ M}(\chi_\Omega-\chi_{\Omega_0})h\circ\phi\,\mathrm d\mathcal H^n,
\end{equation*}
on
$$
\mathcal C_T=\left\{\text{Caccipoli set}\,\Omega\subset M:\Omega\Delta\Omega_0\Subset\phi^{-1}\left((-T,T)\right)\right\}.
$$
For the minimizing problem of the functional $\mathcal A^h$ on $\mathcal C_T$, we have the following existence result, which is just Proposition 2.1 in \cite{Z2020width}.
\begin{lemma}\label{Lem: existence of mu bubble}
Assume that $\pm T$ are regular values of $\phi$ and also the function $h$ satisfies
\begin{equation}\label{Eq: h blow up}
\lim_{t\to -T}h(t)=+\infty\quad \text{\rm and}\quad\lim_{t\to T}h(t)=-\infty,
\end{equation}
then
there exists a smooth minimizer $\hat\Omega$ in $\mathcal C_T$ for $\mathcal A^h$.
\end{lemma}

For our application we only focus on those functions constructed below rather than general ones.
\begin{lemma}\label{Lem: function h epsilon}
For any $\epsilon\in(0,1)$, there is a function $$h_\epsilon:\left(-\frac{1}{n\epsilon},\frac{1}{n\epsilon} \right)\to\mathbf R$$
such that
\begin{itemize}
\item[(1)] $h_\epsilon$ satisfies
\begin{equation*}
\frac{n}{n-1}h_\epsilon^2+2h_\epsilon'=n(n-1)\epsilon^2\quad \text{on}\quad \left(-\frac{1}{n\epsilon},-\frac{1}{2n}\right]\cup\left[\frac{1}{2n}, \frac{1}{n\epsilon}\right)
\end{equation*}
and there is a universal constant $C=C(n)$ so that
\begin{equation*}
\sup_{-\frac{1}{2n}\leq t\leq \frac{1}{2n}} \left|\frac{n}{n-1}h_\epsilon^2+2h_\epsilon'\right|\leq C\epsilon.
\end{equation*}
\item[(2)] $h_\epsilon'<0$ and
$$
\lim_{t\to\mp \frac{1}{n\epsilon}} h_\epsilon(t)=\pm\infty.
$$
\item[(3)] As $\epsilon\to 0$, $h_\epsilon$ converge smoothly to $0$ on any closed interval.
\end{itemize}
\end{lemma}
\begin{proof}
Let
$$
h_\epsilon^+:\left(-\frac{1}{n\epsilon},+\infty\right),\quad t\mapsto (n-1)\epsilon\coth\left(\frac{n\epsilon t+1}{2}\right)
$$
and
$$
h_\epsilon^-:\left(-\infty,\frac{1}{n\epsilon}\right),\quad t\mapsto-(n-1)\epsilon\coth\left(\frac{-n\epsilon t+1}{2}\right).
$$
Through a straightforward calculation we see that $h_\epsilon^+$ and $h_\epsilon^-$ are solutions to the equation
$$
\frac{n}{n-1}h^2+2h'=n(n-1)\epsilon^2.
$$
We now glue $h_\epsilon^+$ and $h_\epsilon^-$ to obtain the desired function $h_\epsilon$. Fix a nonnegative smooth function $\bar\eta$ with compact support contained in $(-\frac{1}{2n},\frac{1}{2n})$ and define
$$
\eta(t)=\left(\int_{-\infty}^{+\infty}\bar\eta(s)\,\mathrm ds\right)^{-1}\int_{-\infty}^t\bar\eta(s)\,\mathrm ds.
$$
Clearly, $\eta$ is smooth with $\eta'\geq 0$ and $0\leq\eta\leq 1$. Furthermore, it satisfies $\eta\equiv 0$ on $(-\infty,-\frac{1}{2n}]$ and $\eta\equiv 1$ on $[\frac{1}{2n},+\infty)$. Denote
$$
h_\epsilon=(1-\eta)h_\epsilon^++\eta h_\epsilon^-.
$$
The proof is now completed by verifying listed properties one by one.
\end{proof}

For convenience we denote
\begin{equation*}
\mathcal A^\epsilon(\Omega)=\mathcal H^{n-1}(\partial^*\Omega)-\int_{ M}(\chi_\Omega-\chi_{\Omega_0})h_\epsilon\circ\phi\,\mathrm d\mathcal H^n,
\end{equation*}
and
$$
\mathcal C_\epsilon=\left\{\text{Caccipoli set}\,\Omega\subset M:\Omega\Delta\Omega_0\Subset\phi^{-1}\left(- \frac{1}{n\epsilon},\frac{1}{n\epsilon}\right)\right\}.
$$

Combining Lemma \ref{Lem: existence of mu bubble} and the Sard's theorem we conclude that
\begin{proposition}\label{Prop: existence}
For almost every $\epsilon\in(0,1)$, there is a smooth minimizer $\hat\Omega_{\epsilon}$ in $\mathcal C_{\epsilon}$ for the functional $\mathcal A^{\epsilon}$.
\end{proposition}

\section{Ricci curvature estimates}
The goal of this section is to prove several Ricci curvature estimates for limit manifolds under some convergence procedure. In the following, we will denote $B_{g}(p,r)$ to be the geodesic ball of radius $r$ centered at a point $p$ with respect to any given metric $g$.
First let us recall the following definition.
\begin{definition}\label{Defn: smooth convergence}
For a sequence of pointed complete Riemannian manifolds $( M_k, g_k, p_k)$, we say that it converges to some pointed complete Riemannian manifold $( M, g, p)$ in the sense of pointed smooth topology if for any $r>0$ there exist an open region $\Omega\subset  M$ containing $B_{ g}( p,r)$ and smooth maps $\Phi_k:\Omega\to M_k$ with $\Phi_k(p)=p_k$ so that
\begin{itemize}
\item $\Phi_k:\Omega\to \Phi_k(\Omega)$ is a diffeomorphism for all $k$ sufficiently large;
\item $\Phi_k^*(g_k)$ converges to $g$ smoothly in any coordinate of $\Omega$.
\end{itemize}
\end{definition}

The idea for the proposition below comes from Kazdan's work in \cite{Kazdan82}.
\begin{proposition}\label{Prop: Ricci flat limit}
For $3\leq n\leq 7$, let $( M_k, g_k)$ be a sequence of connected orientable closed Riemannian $n$-manifolds such that
\begin{itemize}
\item $ M_k$ admits a smooth map $ f_k: M_k\to T^n$ with nonzero degree;
\item With some point $ p_k\in M_k$ the pointed sequence $( M_k,g_k, p_k)$ converges to some pointed complete Riemannian manifold $( M, g, p)$ in the sense of pointed smooth topology;
\item $ R( g_k)\geq -\delta_k\to 0$ as $k\to\infty$ and there is a positive constant $r_0$ such that $R( g_k)$ is nonnegative outside $B_{ g_k}( p_k,r_0)$.
\end{itemize}
Then the limit $( M, g)$ is Ricci flat.
\end{proposition}
\begin{proof}
First let us show that the scalar curvature $R(g)$ vanishes everywhere. Notice that the scalar curvature $R(g)$ is nonnegative from our assumption, we show that it is identical to zero. Otherwise $R( g)$ must be positive at some point $q$ in $ M$. Fix a positive constant $r>r_0$ such that $ q$ is contained in $B_{ g}( p,r)$. For smooth metric $g$ we denote $\Delta_{c,g}$ to be the conformal Laplace operator with respect to $g$. It is clear that the first Neumann eigenvalue of $\Delta_{c,g}$ on $B_g(p,r)$ given by
$$
\mu_1(B_{ g}( p,r), \Delta_{c,g})=\inf_{\psi\in C^\infty-\{0\}}\frac{\int_{B_{ g}( p,r)}|\nabla_{ g}\psi|^2+c_nR( g)\psi^2\,\mathrm d\mu_{ g}}{\int_{B_{ g}( p,r)}\psi^2\,\mathrm d\mu_{ g}}
$$
is positive, where $c_n=\frac{n-2}{4(n-1)}$. From the pointed smooth convergence of $(M_k,g_k,p_k)$, we see that the first Neumann eigenvalue $\mu_1(B_{g_k}(p_k,r), \Delta_{c,g_k})$ is also positive when $k$ is large enough. Combined with the nonnegativity of scalar curvature $R(g_k)$ outside ball $B_{g_k}(p_k,r_0)$, we conclude that $( M_k, g_k)$ must be Yamabe-positive, which is impossible from the work in \cite{SY1979} since each $ M_k$ admits a smooth map to $T^n$ with nonzero degree.

Next we prove $\Ric( g)\equiv 0$. Otherwise $\Ric( g)$ is nonzero at a point $q$. We pick up $r$ in the same way as above. According to \cite[Lemma 3.3]{Kazdan82} we can find a small positive constant $\tau$ and a nonnegative cut-off function $\eta$ supported around point $q$ such that the first Neumann eigenvalue $\mu_1(B_{{ g}}( p,r), \Delta_{c,g'})$ is positive with $ g'= g+\tau\eta\Ric( g)$. By definition there are an open region $\Omega\supset B_{g}(p,r)$ and diffeomorphisms $\Phi_k:\Omega\to \Omega_k\subset M_k$ for $k$ large enough. Define
$$
g_k'=g_k+\tau(\eta\circ\Phi_k^{-1})\Ric(g_k).
$$
From the pointed smooth convergence it follows that the first Neumann eigenvalue $\mu_1(B_{ g_k}( p_k,r), \Delta_{c,g_k'})$ is positive for $k$ large enough, which leads to a contradiction just as before.
\end{proof}

In the following, we prove more Ricci curvature estimates similar to the above one, which will be used in the proof for Theorem \ref{Thm: main3}. Before that, let us make some necessary preparations.

Let $(M,g)$ be a closed orientable Riemannian manifold, which admits a smooth map $f:M\to \mathbf S^2\times T^{n-2}$ with nonzero degree. Denote
$$\pi:\mathbf S^2\times T^{n-2}\to \mathbf S^2$$
to be the canonical projection. We introduce the {\it partial spherical $2$-systole} with respect to map $f$ as following
\begin{equation*}
 \sys'_2( M, g,f):=\inf\left\{\area(\mathbf S^2,i^* g)\left|\begin{array}{c}
 \text{$i:\mathbf S^2\to M$ smooth such that}\\
 \text{$(\pi\circ f\circ i)_*([\mathbf S^2])\neq 0$}
 \end{array}\right.\right\}.
\end{equation*}
From this definition it is clear that $\sys_2'(M,g,f)$ is no less than $\sys_2(M,g)$.

Recall that the author proved the following result in \cite{Z2020}:
\begin{theorem}\label{Thm: 2 systole n manifold}
For $3\leq n\leq 7$, let $(M^n,g)$ be a connected orientable closed Riemannian manifold with positive scalar curvature, which admits a smooth map $f:M\to \mathbf S^2\times T^{n-2}$ with nonzero degree. Then there exists an embedded sphere $\Sigma$ such that $(\pi\circ f)_*([\Sigma])\neq 0$ and
\begin{equation}\label{Eq: 2 systole n manifold}
\inf_M R(g)\cdot\sys_2'(M,g,f)\leq 8\pi.
\end{equation}
The equality holds if and only if the universal covering of $(M,g)$ is isometric to the standard Riemannian product $\mathbf S^2\times \mathbf R^{n-2}$ up to scaling.
\end{theorem}
\begin{remark}
Despite of the slightly different statement from the original one in \cite{Z2020}, it can be quickly deduced from the exactly same proof.
\end{remark}
\begin{remark}\label{Remark: general regularity}
From the general regularity result in \cite{NS1993} by Nathan Smile, the estimate \eqref{Eq: 2 systole n manifold} also holds when dimension $n=8$.
\end{remark}

In the same spirit of Proposition \ref{Prop: Ricci flat limit} we establish the following result:
\begin{proposition}\label{Prop: scalar 2 limit}
For $3\leq n\leq 7$, let $( M_k, g_k)$ be a sequence of orientable connected closed Riemannian $n$-manifolds such that
\begin{itemize}
\item $M_k$ admits a smooth map $f_k:M_k\to \mathbf S^2\times T^{n-2}$ with nonzero degree;
\item For all $k$ there holds $\sys'_2( M_k, g_k,f_k)\geq 4\pi$;
\item With some $ p_k\in M_k$ the pointed sequence $( M_k,g_k, p_k)$ converges to some complete Riemannian manifold $( M, g, p)$ in the sense of smooth topology;
\item $ R( g_k)\geq \delta_k\to 2$ as $k\to\infty$ and there is a positive constant $r_0$ such that $R( g_k)\geq 2$ outside $B_{ g_k}( p_k,r_0)$.
\end{itemize}
Then the limit $(M,g)$ has constant scalar curvature $2$ and nonnegative Ricci curvature.
\end{proposition}
The proof is quite similar to that of Proposition \ref{Prop: Ricci flat limit} except that we focus on a new elliptic operator from the following warped product construction. Given a complete Riemannian manifold $(M^n,g)$ and also a smooth positive function $u:M\to \mathbf R$, we can define
\begin{equation*}
\tilde M=M\times \mathbf S^1\quad \text{and}\quad \tilde g=g+u^2\mathrm ds^2,
\end{equation*}
where $\mathbf S^1$ is the unit circle and $s$ is the arc length on it. Through a simple calculation we know that the scalar curvature of this new manifold is
$$
R(\tilde g)=R(g)-\frac{2\Delta_gu}{u}.
$$
The benefit of this construction is that we can increase the scalar curvature without changing the partial spherical $2$-systole. In fact, if the manifold $M$ admits a smooth map $f:M\to \mathbf S^2\times T^{n-2}$ with nonzero degree, it holds
$$
\sys_2'(M,g,f)=\sys_2'(\tilde M,\tilde g,f\times \id).
$$
So this provides us an ideal way to deform the given metric $g$.
Notice that the existence of a smooth positive function $u$ on $M$ such that $R(\tilde g)$ has a positive lower bound is closely related to the positivity of the elliptic operator $$-\Delta_g+\frac{1}{2}R(g).$$
We put our attention on this operator instead of the conformal Laplacian.

For those readers who want to skip the tedious details, we give a quick sketch as follows. The point is that the partial spherical $2$-systole condition on $(M_k,g_k)$ combined with Theorem \ref{Thm: 2 systole n manifold} forbids any possible warped product construction to increase the scalar curvature to be greater than $2$, which leads to the same restriction to $(M,g)$ after we take the limit. Denote
\begin{equation}
L_g:=-\Delta_g+\frac{1}{2}(R(g)-2).
\end{equation}
This means that $L_g$ cannot be a positive operator in any compact region and so the sclar curvature $R(g)$ must be identical to 2 from our assumption $R(g)\geq 2$. If the Ricci curvature of the metric $g$ is negative somewhere, we are able to deform the metric along Ricci direction to obtain the positivity of operator $L_g$ without decreasing the metric. Back to the level of $(M_k,g_k)$, we can find a warped product construction with scalar curvature greater than $2$ but the partial spherical $2$-systole still no less than $4\pi$. This is impossible due to Theorem \ref{Thm: 2 systole n manifold}.

The full details of the proof for Proposition \ref{Prop: scalar 2 limit} is presented below.
\begin{proof}[Froof for Proposition \ref{Prop: scalar 2 limit}]
From our assumption it is clear that the scalar curvature $R( g)$ is no less than $2$ everywhere. Assume otherwise there is a point $q$ such that $R( g)$ is greater than $2$. We take a positive constant $r>r_0$ such that $ q$ is contained in $B_{ g}( p,r)$. It is clear that the first Neumann eigenvalue of $L_g$ on $B_g(p,r)$ defined by
$$
\mu_1(B_{ g}( p,r),L_g)=\inf_{\psi\in C^\infty-\{0\}}\frac{\int_{B_{ g}( p,r)}|\nabla_{ g}\psi|^2+\frac{1}{2}(R( g)-2)\psi^2\,\mathrm d\mu_{ g}}{\int_{B_{ g}( p,R)}\psi^2\,\mathrm d\mu_{ g}}
$$
is positive. From the pointed smooth convergence the first Neumann eigenvalue $\mu_1(B_{ g_k}( p_k,r),L_{g_k})$ is also positive for $k$ large enough. Since $R( g_k)$ is no less than 2 outside $B_{ g_k}( p_k,r_0)$, the first eigenvalue $\lambda_{1,k}$ of $L_{g_k}$ on $M_k$ is positive. Let $v_k$ be the corresponding first eigenfunction, then we have
$$
-\Delta_{ g_k} v_k+\frac{1}{2}(R( g_k)-2) v_k=\lambda_{1,k} v_k.
$$
Define $\tilde M_k=M_k\times \mathbf S^1$ and $\tilde g_k= g_k+ v_k^2\mathrm ds^2$. It follows
$$
R(\tilde g_k)=R( g_k)-\frac{\Delta_{ g_k} v_k}{ v_k}=2+2\lambda_{1,k}>2
$$
and
$$\sys_2'(\tilde M_k,\tilde g_k,f_k\times \id)=\sys_2'( M_k, g_k,f_k)\geq 4\pi.$$
This leads to a contradiction to Theorem \ref{Thm: 2 systole n manifold}.

Next we prove $\Ric(g)\geq 0$. Assume otherwise $\Ric(g)$ is negative along some tangent vector at a point $q$. Then we can find a unit vector $v\in T_qM$ with respect to $g$ such that
$
\Ric_{g}(v)=cv
$ with $c$ to be a negative constant. Extend $v$ to a unit vector field $V$ on some neighborhood of $q$ and denote $\omega$ to be the corresponding dual $1$-form. From continuity $\Ric(V,V)$ takes negative values around $q$. Let $h=\Ric(V,V)\,\omega\otimes\omega$. Clearly we have
$$
\langle h,\Ric(g)\rangle_g(q)=c^2>0.
$$
Therefore we can assume that $\langle h,\Ric(g)\rangle_g$ is positive on a neighborhood of $q$. Take a nonnegative cut-off function $\eta$ supporting on this neighborhood and define
$
g_t=g-2t\eta h.
$
Taking the derivative of the scalar curvature function with respect to $ g_t$, we obtain
$$
\left.\frac{\partial}{\partial t}\right|_{t=0}R( g_t)=-\Delta_{ g}\tr_{ g}h+\Div_{ g}\Div_{ g}h-\langle h,\Ric( g)\rangle_{ g}.
$$
Take $r>r_0$ such that $B_{g}(p,r)$ contains the above neighborhood of $q$.
With a similar calculation as in the proof of \cite[Theorem 1.7]{LM2019} we conclude that
\begin{equation}\label{Eq: derivative newmann eigenvalue}
\left.\frac{\mathrm d}{\mathrm dt}\right|_{t=0}\mu_1(B_{ g}( p,r),L_{ g_t})
=\vol_{ g}(B_{ g}( p,r))^{-1}\int_{ M}\eta\langle h,\Ric(g)\rangle_g\mathrm d\mu_{ g}
>0.
\end{equation}
So we can pack up a positive constant $\tau$ such that $\mu_1(B_{g}(p,r),L_{g_{\tau}})>0$. Let $\Omega\supset B_g (p,r)$ and $\Phi_k:\Omega\to \Phi_k(\Omega)$ be the diffeomorphisms from Definition \ref{Defn: smooth convergence}. Define
$$
g_{k,\tau}=g_k-2\tau(\Phi_k^{-1})^*(\eta h).
$$
For $k$ large enough we have $\mu_1(B_{g_k}(p_k,r),g_{k,t_0})>0$. As before, we can define a new manifold $(\tilde M_k,\tilde g_k)$ with $\tilde M_k= M_k\times \mathbf S^1$ and $\tilde g_k= g_{k,\tau}+v_k^2\mathrm ds^2$ such that $R(\tilde g_k)>2$. However it follows from our construction that $g_{k,\tau}\geq g_k$ as quadratic forms. Therefore
$$
\sys_2'(\tilde M_k,\tilde g_k,f_k\times \id)=\sys_2'(M_k,g_{k,\tau},f_k)\geq \sys_2'(M_k,g_{k},f_k)\geq 4\pi.
$$
This leads to a contradiction to Theorem \ref{Thm: 2 systole n manifold}.
\end{proof}

If all manifolds in the sequence have a warped product structure that can be preserved to the limit, then we can show the following result.
\begin{proposition}\label{Prop: vanish normal ricci limit}
For $3\leq n\leq 7$, let $(\bar M_k,\bar g_k)$ be a sequence of orientable connected closed Riemannian $n$-manifolds. Assume that
\begin{itemize}
\item $\bar M_k= M_k\times \mathbf S^1$ and $\bar g_k= g_k+ u_k^2\mathrm ds^2$, where $g_k$ is a smooth metric on $ M_k$, $ u_k:M_k\to \mathbf R$ is a positive function and $s$ is the arc length of the unit circle $\mathbf S^1$;
\item there is a smooth map $f_k: M_k\to \mathbf S^2\times T^{n-3}$ with nonzero degree and $\sys'_2( M_k, g_k,f_k)\geq 4\pi$ for all $k$;
\item With some point $p_k\in M_k$ the pointed sequence $( M_k, g_k, p_k)$ converges to a pointed complete Riemannian manifold $( M, g, p)$ in the sense of pointed smooth topology and also $ u_k$ converges smoothly to a positive function $ u: M\to \mathbf R$ on each compact subset;
\item $R(\bar g_k)\geq \delta_k\to 2$ as $k\to\infty$ and there is a universal constant $r_0$ such that the scalar curvature $R(\bar g_k)\geq 2$ outside $B_{\bar g_k}(\bar p_k,r_0)$ with $\bar p_k=( p_k,\theta)$ for a fixed $\theta\in\mathbf S^1$.
\end{itemize}
Then $(\bar M,\bar g)$ has constant scalar curvature $2$ and nonnegative Ricci curvature, where $\bar M= M\times\mathbf S^1$ and $\bar g= g+ u^2\mathrm ds^2$. Moreover, it holds $$\Ric_{\bar g}(\partial_s,\partial_s)\equiv 0.$$
\end{proposition}

The key observation for above proposition is that the deformation along normal Ricci direction only makes a change on lapse function and this causes no effect on the partial spherical $2$-systole.
\begin{proof}
From our assumption the manifolds $(\bar M_k,\bar g_k,\bar p_k)$ converge to $(\bar M,\bar g,\bar p)$ in the sense of pointed smooth topology and so the first part of the conclusion comes directly from Proposition \ref{Prop: scalar 2 limit}.
In the following we prove $\Ric_{\bar g}(\partial_s,\partial_s)\equiv 0$. If this is not the case, then we can pick up a point $\bar q=( q,\theta)$ such that $\Ric_{\bar g}(\partial_s,\partial_s)$ is nonzero at $\bar q$ due to the $\mathbf S^1$-invariance of $(\bar M,\bar g)$. Choose a positive constant $r>r_0$ such that $p$ is contained in $B_{ g}( p,r)$ and take a nonnegative cut-off function $\eta$ on $ M$ taking value $1$ around $ q$ and vanishing outside $B_{ g}( p,r)$. For small $t$ we can define
$$
\bar g_t=g+u_t^2\mathrm ds^2\quad\text{with}\quad u_t:= u-2t\eta\Ric_{\bar g}(\partial_s,\partial_s).
$$
Notice that $u_t$ is actually a positive function on $M$ due to the $\mathbf S^1$-invariance of $(\bar M,g)$.  It is clear that all hypersurfaces $M\times\{\theta'\}$ with $\theta'\in \mathbf S^1$ are totally geodesic in $(\bar M,\bar g)$ since they can be realized as the fixed set of a reflection. Therefore the mixed Ricci curvature $\Ric(\partial_i,\partial_s)$ vanishes from the Codazzi equation.
Repeating the calculation in \eqref{Eq: derivative newmann eigenvalue}, we conclude that there is a small positive constant $\tau$ such that the first Neumann eigenvalue $\mu_1(B_{\hat g}(\hat p,R)\times \mathbf S^1,L_{\bar g_{\tau}})$ is positive. Take $\Omega\supset B_{g}(p,r)$ and denote $\Phi_k:\Omega\to \Omega_k\subset M_k$ to be the diffeomorphisms from Definition \ref{Defn: smooth convergence}. Let us define
$$
\bar g_{k,\tau}= g_k+\left( u_k-2\tau(\eta\circ\Phi_k^{-1})\Ric_{\bar g_k}(\partial_s,\partial_s)\right)\mathrm ds^2.
$$
From the pointed smooth convergence the first Neumann eigenvalue $$\mu_1(B_{ g_k}( p_k,r)\times \mathbf S^1,L_{\bar g_{k,\tau}})$$
is positive for sufficiently large $k$. As before, we can define a new manifold $(\tilde M_k,\tilde g_k)$ with $\tilde M_k=\bar M_k\times \mathbf S^1$ and $\tilde g_k=\bar g_{k,\tau}+v_k^2\mathrm ds^2$ such that $R(\tilde g_k)>2$. Notice that
\begin{equation*}
\begin{split}
\sys_2'(\tilde M_k,\tilde g_k,f_k\times \id\times \id)&=\sys_2'(\bar M_k,\bar g_{k,\tau},f_k\times \id)\\
&=\sys_2'(M_k,g_k,f_k)\geq 4\pi,
\end{split}
\end{equation*}
we obtain a contradiction to Theorem \ref{Thm: 2 systole n manifold}.
\end{proof}

\section{Proof for main theorems}
In this section, we give a detailed proof for our main theorems.
\subsection{Proof for Theorem \ref{Thm: main1}}
In this subsection, $(M^n,g)$ always denotes a connected orientable complete open Riemannian manifold with nonnegative scalar curvature, which admits a smooth proper map
$$f:M\to T^{n-1}\times \mathbf R$$
with nonzero degree. We also denote
$$\pi_1:T^{n-1}\times \mathbf R\to T^{n-1}\quad\text{and}\quad\pi_2:T^{n-1}\times \mathbf R\to \mathbf R$$
to be the canonical projection maps. It follows from the Sard's theorem that there is a constant $r_0$ such that $\Sigma=(\pi_2\circ f)^{-1}(r_0)$ is a closed embedded separable hypersurface in $M$ with
\begin{equation}\label{Eq: Sigma0 nonzero degree}
(\pi_1\circ f)_*([\Sigma])=(\deg f)[T^{n-1}].
\end{equation}
Clearly, $\Sigma$ separates $M$ into two unbounded parts
$$M_+=\{\pi_2\circ f>r_0\}\quad \text{and}\quad M_-=\{\pi_2\circ f<r_0\}.$$
As a result, a surjective signed distance function to $\Sigma$ is well-defined.
Below we just take the same notation as in Section \ref{Sec: 2} to give a proof for Theorem \ref{Thm: main1}.
\begin{proof}[Proof for Theorem \ref{Thm: main1}]
From Proposition \ref{Prop: existence} we can take positive constants $\epsilon_k\to 0$ as $k\to \infty$ such that we can construct a smooth minimizer $\hat\Omega_k$ in $\mathcal C_{\epsilon_k}$ for functional $\mathcal A^{\epsilon_k}$. It follows from a direct comparison that $$\area(\partial\hat\Omega_k)\leq \area(\Sigma).$$
Since boundary $\partial\hat\Omega_k$ is homologous to the hypersurface $\Sigma$, combined with \eqref{Eq: Sigma0 nonzero degree} we can pick up some connected component $\hat\Sigma_k$ of $\partial\hat\Omega_k$ such that
\begin{equation}\label{Eq: hat sigma k nonzeo degree}
(\pi_1\circ f)_*([\hat\Sigma_k])=c[T^{n-1}],\quad c\neq 0.
\end{equation}
Notice that hypersurfaces $\hat\Sigma_k$ are $h_{\epsilon_k}$-minimizing boundaries with uniform area bound, so we can apply the curvature estimate in \cite[Theorem 3.6]{ZZ2018} to $\hat\Sigma_k$. After passing to a subsequence $\hat\Sigma_k$ converges to an area-minimizing boundary $\hat\Sigma$ smoothly in locally graphical sense with multiplicity one. From the stability of $\hat\Sigma_k$ we can find a nonnegative constant $\lambda_{1,k}$ and a positive function $\hat u_k:\hat\Sigma_k\to \mathbf R$ such that
\begin{equation}
-\hat\Delta_{k}\hat u_k-\left(\Ric(\nu_k,\nu_k)+|A_k|^2-\partial_{\nu_k} (h_{\epsilon_k}\circ\phi)\right)\hat u_k=\lambda_{1,k}\hat u_k,
\end{equation}
where $\hat\Delta_k$ is the Laplace-Beltrami operator of $\hat\Sigma_k$ with induced metric $\hat g_k$, $\nu_k$ is the outward unit normal vector field of $\hat\Sigma_k$ with respect to $\hat\Omega_k$ and $A_k$ is the corresponding second fundamental form. Let $s$ be the arc length of the unit circle $\mathbf S^1$, then we define $\bar\Sigma_k=\hat\Sigma_k\times \mathbf S^1$ and $\bar g_k=\hat g_k+\hat u_k^2\mathrm ds^2$. Direct calculation gives
\begin{equation}\label{Eq: scalar bar gk}
\begin{split}
R(\bar g_k)&=R(\hat g_k)-\frac{2\hat\Delta_k \hat u_k}{\hat u_k}\\
&\geq R(g)+|\mathring A_k|^2+\left(\frac{n}{n-1}h_{\epsilon_k}^2+2h'_{\epsilon_k}\right)\circ\phi+2 \lambda_{1,k}.
\end{split}
\end{equation}
From \eqref{Eq: hat sigma k nonzeo degree} we see that $\bar\Sigma_k$ admits a smooth map to $T^n$ with nonzero degree. Combined with the nonnegativity of $R(g)$ and Lemma \ref{Lem: function h epsilon}, it follows from \cite[Corollary 2]{SY1979} that $\hat\Sigma_k$ must have nonempty intersection with the fixed compact subset $K=\phi^{-1}([-\frac{1}{2n},\frac{1}{2n}])$ and $0\leq \lambda_{1,k}\leq C(n)\epsilon_k$. Pick up a point $\hat p_k\in\hat\Sigma_k\cap K$ for each $k$ and normalize $\hat u_k(\hat p_k)=1$. Up to subsequence the locally graphical convergence with multiplicity one implies that the pointed sequence $(\hat \Sigma_k,\hat p_k)$ converges to $(\hat\Sigma',\hat p)$ in the pointed smooth topology. Here $\hat\Sigma'$ is one connected component of $\hat\Sigma$ and $\hat p$ is a point on $\hat\Sigma'$.
From Harnack inequality $\hat u_k$ converges smoothly to a function $\hat u:\hat\Sigma'\to \mathbf R$ with $\hat u(\hat p)=1$ after passing to a further subsequence.

Similarly, we define $\bar\Sigma=\hat\Sigma'\times \mathbf S^1$ and $\bar g=\hat g+\hat u^2\mathrm dt^2$ with $\hat g$ to be the induced metric of $\hat\Sigma'$. From previous construction it is quick to see that $(\bar \Sigma_k,\bar g_k,\bar p_k)$ converges to $(\bar \Sigma,\bar g,\bar p)$ in the pointed smooth topology, where $\bar p_k=(\hat p_k,\theta)$ and $\bar p=(\hat p,\theta)$ for a fixed $\theta$ in $\mathbf S^1$. From \eqref{Eq: scalar bar gk} we conclude that the scalar curvature $R(\bar g_k)$ is nonnegative outside $(\hat\Sigma_k\cap K)\times \mathbf S^1$. Given the curvature estimate and the uniform area bound for $\hat\Sigma_k$, this yields the existence of a positive constant $R_0$ such that $R(\bar g_k)\geq 0$ outside $B_{\bar g_k}(\bar p_k,R_0)$. We point out that \eqref{Eq: scalar bar gk} also implies $R(\bar g_k)\geq -C\epsilon_k\to 0$ as $k\to\infty$. Now it follows from Proposition \ref{Prop: Ricci flat limit} that $(\bar\Sigma,\bar g)$ is Ricci-flat.

Through investigating the variation induced by the isometric $\mathbf S^1$-action of $(\bar\Sigma,\bar g)$ we can obtain $\hat\Delta\hat u=0$, where $\hat\Delta$ denotes the Laplace-Beltrami operator of $\hat\Sigma'$. We show that the function $\hat u$ has to be a positive constant.

Denote $\hat v=\log\hat u$, then we have
\begin{equation}\label{Eq: equation hat v}
\hat\Delta\hat v=-|\hat\nabla\hat v|^2.
\end{equation}
For any $r>0$ we can take a smooth cut-off function $\eta_r$ such that $0\leq \eta_r\leq 1$, $\eta_r\equiv 1$ in $B_{\hat g}(\hat p,r)$, $\eta_r\equiv 0$ outside $B_{\hat g}(\hat p,2r)$ and $|\hat\nabla\eta_r|\leq Cr^{-1}$. Multiplying \eqref{Eq: equation hat v} by $\eta_r^2$ and integrating by parts, we see
\begin{equation*}
\begin{split}
-\int_{\hat\Sigma'}\eta_r^2|\hat\nabla\hat v|^2\mathrm d\mu_{\hat g}&=-2\int_{\hat\Sigma'}\eta_r\hat\nabla\eta_r\cdot\hat\nabla\hat v\,\mathrm d\mu_{\hat g}\\
&\geq -2\int_{\hat\Sigma'}|\hat\nabla\eta_r|^2\mathrm d\mu_{\hat g}-\frac{1}{2}\int_{\hat\Sigma'}\eta_r^2|\hat\nabla\hat v|^2\mathrm d\mu_{\hat g}.
\end{split}
\end{equation*}
This implies
\begin{equation*}
\int_{B_{\hat g}(\hat p,r)}|\hat\nabla\hat v|^2\mathrm d\mu_{\hat g}\leq 4\int_{\hat\Sigma'}|\hat\nabla\eta_r|^2\mathrm d\mu_{\hat g}\leq 4C^2r^{-2}\area(\hat\Sigma').
\end{equation*}
Notice that $\area(\hat\Sigma')\leq \area(\Sigma)$, by letting $r\to+\infty$ we obtain that $\hat u$ is a constant. Therefore $\hat\Sigma'$ is a complete Ricci-flat Riemannian manifold with finite area. It is a well-known fact that an open complete Riemannian manifold with nonnegative Ricci curvature has infinite volume (see for instance \cite[Theorem 7]{Yau1976}). So $\hat\Sigma'$ has to be a closed hypersurface in $M$. From the connectedness of $\hat\Sigma_k$ and the local graphical convergence, we conclude $\hat\Sigma=\hat\Sigma'$ and that $\hat\Sigma_k$ is a graph over $\hat\Sigma$ when $k$ is sufficiently large. In particular, $\hat\Sigma$ admits a smooth map to $T^{n-1}$ with nonzero degree. Since $\hat\Sigma$ is an area-minimizing boundary in a Riemannian manifold with nonnegative scalar curvature, it must be a flat $(n-1)$-torus. Then the standard foliation argument as in \cite[Proposition 3.4]{Z2020} yields that $M$ splits as the Riemannian product $T^{n-1}\times \mathbf R$.
\end{proof}

\subsection{Proof for Theorem \ref{Thm: main2}}
In this subsection, we denote $(M^3,g)$ to be a complete connected open Riemannian manifold with positive scalar curvature and non-trivial second homotopy group. From the lifting theorem for covering spaces we only need to deal with the case when $M$ is simply connected and this is always assumed in the rest of this section.

Since $\pi_2(M)$ is non-trivial, it follows from the sphere theorem that there is an embedded sphere $\Sigma$ in $M$ representing a nontrivial element in $\pi_2(M)$. The Hurewicz theorem yields that the sphere $\Sigma$ is also homologically non-trivial. Combined with the simply-connectedness of $M$ it must separate $M$ into two unbounded components. Notice that there is a well-defined surjective signed distance function to $\Sigma$ and so the discussion in Section \ref{Sec: 2} is also valid here.

Before we prove our main theorem we present the following description for stable minimal surfaces in $3$-manifolds with uniformly positive scalar curvature, which plays a crucial role in our argument.
\begin{lemma}\label{Lem: stable min}
Let $(M^3,g)$ be an orientable complete manifold with uniformly positive scalar curvature. If $\Sigma$ is a complete two-sided stable immersed minimal surface in $M$, then it is a sphere.
\end{lemma}
\begin{proof}
It follows from \cite[Theorem 3]{FS1980} and \cite[Proposition C.1]{CCE2016} that $\Sigma$ must be a sphere or a plane. We show that the latter case cannot happen. Otherwise from stability there is a positive function $u$ such that
$$
-\Delta_{\Sigma}u-\frac{1}{2}(R(g)-R(g_\Sigma)+|A|^2)u=0,
$$
where $g_\Sigma$ is the induced metric of $\Sigma$, $\Delta_\Sigma$ denotes the corresponding Laplace-Beltrami operator and $A$ is the second fundamental form of $\Sigma$. Define the metric $\tilde g=g_\Sigma+u^2\mathrm ds^2$ on $\tilde M=\Sigma\times \mathbf S^1$. Clearly, $(\tilde M,\tilde g)$ is complete with uniformly positive scalar curvature since we have
$$
R(\tilde g)=R_\Sigma-{2u^{-1}\Delta_\Sigma u}=R(g)+|A|^2.
$$
This then contradicts to the quadratic decay theorem in \cite{Gromov2018}.
\end{proof}

Now we are ready to prove Theorem \ref{Thm: main2}.
\begin{proof}[Proof for Theorem \ref{Thm: main2}]
We only need to prove when $\inf R(g)>0$. Take a sequence of positive constants $\epsilon_k$ such that $\epsilon_k\to 0$ as $k\to\infty$ such that we can construct a smooth minimizer $\hat\Omega_k$ in $\mathcal C_{\epsilon_k}$ for functional $\mathcal A^{\epsilon_k}$. Since $\partial\hat\Omega_k$ is homologous to $\Sigma$, we can choose one of its components to be homologically nontrivial, denoted by $\hat\Sigma_k$. It follows from the stability and the Gauss equation that
\begin{equation*}
\begin{split}
\int_{\hat\Sigma_k}|\nabla\psi|^2\,\mathrm d\sigma_g
&\geq \int_{\hat\Sigma_k}\left(\Ric(\nu,\nu)+|A|^2-\partial_\nu(h_{\epsilon_k}\circ\phi) \right)\psi^2\,\mathrm d\sigma_g\\
&\geq\frac{1}{2}\int_{\hat\Sigma_k}\left(R(g)-R_{\hat\Sigma_k}+|\mathring A|^2+\left(\frac{3}{2}h_{\epsilon_k}^2+2h_{\epsilon_k}'\right)\circ\phi \right)\psi^2\,\mathrm d\sigma_g.
\end{split}
\end{equation*}
By choosing $\psi\equiv 1$ we conclude
$$
8\pi\geq 4\pi\chi(\hat\Sigma_k)=\int_{\hat\Sigma_k}R_{\hat\Sigma_k}\,\mathrm d\sigma_g\geq \int_{\hat\Sigma_k}R(g)-C\epsilon_k.
$$
For $k$ large enough, surface $\hat\Sigma_k$ is a sphere with
$$\area(\hat\Sigma_k)\leq 8\pi(\inf R(g)-C\epsilon_k)^{-1}.$$
Clearly $\hat\Sigma_k$ represents a nontrivial homotopy class and by letting $k\to\infty$ we obtain
$$\sys_2(M,g)\leq 8\pi(\inf R(g))^{-1}.$$

In the equality case, we can always assume $\inf R(g)=2$ through scaling. If we have $\sys_2(M,g)=4\pi$, then surfaces $\hat\Sigma_k$ must have non-empty intersection with the compact subset $K=\phi^{-1}([-\frac{1}{6},\frac{1}{6}])$. From comparison we have $\area(\hat\Sigma_k)\leq \area(\Sigma)$. Combined with the curvature estimate, up to subsequence $\hat\Sigma_k$ converges smoothly to an area-minimizing boundary $\hat\Sigma$ in locally graphical sense. From Lemma \ref{Lem: stable min} we conclude that $\hat\Sigma$ consists of spherical components. Since $\hat\Sigma_k$ is connected, $\hat\Sigma$ is a sphere and $\hat\Sigma_k$ becomes a graph over $\hat\Sigma$ for $k$ large enough. Therefore $\hat\Sigma$ also represents a nontrivial homotopy class. Then the second variation formula yields
$$
4\pi=\sys_2(M,g)\leq \area(\hat\Sigma)\leq 4\pi.
$$
Therefore $\hat\Sigma$ is area-minimizing in its homotopy class. The rest of the proof is the same as in \cite{BBN2010} and we omit the details.
\end{proof}

\subsection{Proof for Theorem \ref{Thm: main3}}
In this subsection, $(M^n,g)$ always denotes a connected orientable complete open Riemannian manifold with positive scalar curvature, which admits a smooth proper map
$$f:M\to \mathbf S^2\times T^{n-3}\times \mathbf R$$
with nonzero degree. We denote
$$
\pi_1:\mathbf S^2\times T^{n-3}\times \mathbf R\to \mathbf S^2\times T^{n-3}
$$
and
$$
\pi_2: \mathbf S^2\times T^{n-3}\times \mathbf R\to \mathbf R
$$
to be the canonical projection maps. Take $r_0$ such that $\Sigma=(\pi_2\circ f)^{-1}(r_0)$ is a closed embedded separable hypersurface in $M$ and we adopt the same notation as in Section \ref{Sec: 2}. From the definition of $\Sigma$ we have
\begin{equation}
(\pi_1\circ f)_*([\Sigma])=(\deg f)[\mathbf S^2\times T^{n-3}].
\end{equation}

Now we give a proof for Theorem \ref{Thm: main3}. The proof is almost identical to that of Theorem \ref{Thm: main1} while the key difference is that we use Proposition \ref{Prop: vanish normal ricci limit} instead of Proposition \ref{Prop: Ricci flat limit} to obtain the compactness of the limit.

\begin{proof}[Proof for Theorem \ref{Thm: main3}]
As in the proof for Theorem \ref{Thm: main2} we only need to deal with the case $\inf R(g)>0$. Take positive constants $\epsilon_k\to 0$ as $k\to\infty$ such that we can construct a smooth minimizer $\hat\Omega_k$ in $\mathcal C_{\epsilon_k}$ for functional $\mathcal A^{\epsilon_k}$. It follows from direct comparison that $\area(\partial\hat\Omega_k)\leq \area(\Sigma)$. Since $\partial\hat\Omega_k$ is homologous to $\Sigma$, we can choose one component $\hat\Sigma_k$ of $\partial\hat\Omega_k$ such that
\begin{equation}\label{Eq: hat sigma k nonzeo degree 2}
(f_1)_*([\hat\Sigma_k])=c[\mathbf S^2\times T^{n-3}],\quad c\neq 0,
\end{equation}
where $f_1:=\pi_1\circ f$. From the stability of $\hat\Sigma_k$ there exist a nonnegative constant $\lambda_{1,k}$ and a positive function $\hat u_k:\hat\Sigma_k\to \mathbf R$ such that
\begin{equation}
-\hat\Delta_{k}\hat u_k-\left(\Ric(\nu_k,\nu_k)+|A_k|^2-\partial_{\nu_k} (h_{\epsilon_k}\circ\phi)\right)\hat u_k=\lambda_{1,k}\hat u_k,
\end{equation}
where $\hat\Delta_k$ is the Laplace-Beltrami operator of $\hat\Sigma_k$ with induced metric $\hat g_k$, $\nu_k$ is the outward unit normal vector field of $\hat\Sigma_k$ with respect to $\hat\Omega_k$ and $A_k$ is the corresponding second fundamental form. With $s$ to be the arc length of the unit circle $\mathbf S^1$, we define $\bar\Sigma_k=\hat\Sigma_k\times \mathbf S^1$ and $\bar g_k=\hat g_k+\hat u_k^2\mathrm ds^2$. Direct calculation gives
\begin{equation}\label{Eq: scalar bar gk 2}
\begin{split}
R(\bar g_k)&=R(\hat g_k)-\frac{2\hat\Delta_k \hat u_k}{\hat u_k}\\
&\geq R(g)+|\mathring A_k|^2+\left(\frac{n}{n-1}h_{\epsilon_k}^2+2h'_{\epsilon_k}\right)\circ\phi+2 \lambda_{1,k}.
\end{split}
\end{equation}
From \eqref{Eq: hat sigma k nonzeo degree 2} we know that $\bar\Sigma_k$ admits a smooth map to $\mathbf S^2\times T^n$ with nonzero degree. Combining \eqref{Eq: scalar bar gk 2}, Lemma \ref{Lem: function h epsilon} and Theorem \ref{Thm: 2 systole n manifold} we conclude
$$
\sys_2'(\hat \Sigma_k,\hat g_k,f_1|_{\hat\Sigma_k})=\sys_2'(\bar \Sigma_k,\bar g_k,f_1|_{\hat\Sigma_k}\times\id)\leq 8\pi\left(\inf R(g)-C\epsilon_k\right)^{-1}.
$$
Letting $k\to \infty$ and using the fact
$$\sys_2(M,g)\leq \sys_2'(\hat \Sigma_k,\hat g_k,f_1|_{\hat\Sigma_k}),$$
we obtain
$$
\inf R(g)\cdot\sys_{2}(M,g)\leq 8\pi.
$$

In the following we deal with the equality case. Without loss of generality we assume $\inf R(g)=2$ and $\sys_2(M,g)=4\pi$. It then follows
$$
\sys_2'(\hat\Sigma_k,\hat g_k,f_1|_{\hat\Sigma_k})\geq \sys_2(M,g)=4\pi.
$$
From \eqref{Eq: scalar bar gk 2}, Lemma \ref{Lem: function h epsilon} and Theorem \ref{Thm: 2 systole n manifold}, each hypersurface $\hat\Sigma_k$ must have nonempty intersection with the compact subset $K=\phi^{-1}([-\frac{1}{2n},\frac{1}{2n}])$. With exactly the same argument as in the proof of Theorem \ref{Thm: main1}, up to subsequence we can obtain the following:
\begin{itemize}
\item With $\hat p_k\in \hat\Sigma_k\cap K$ the pointed hypersurface $(\hat\Sigma_k,\hat g_k,\hat p_k)$ converges smoothly to an area-minimizing boundary $(\hat\Sigma,\hat g,\hat p)$ in locally graphical sense with multiplicity one;
\item Normalized $\hat u_k(\hat p_k)=1$ the function $\hat u_k$ converges smoothly to a positive function $\hat u:\hat\Sigma\to \mathbf R $ on each compact subset;
\item $R(\bar g_k)\geq \delta_k\to 2$ and there is a universal constant $R_0$ such that the scalar curvature $R(\bar g_k)\geq 2$ outside the geodesic $R_0$-ball $B_{\bar g_k}(\bar p_k,R_0)$ with $\bar p_k=(\hat p_k,\theta)$ for a fixed $\theta\in\mathbf S^1$.
\end{itemize}
Denote $\bar\Sigma=\hat\Sigma\times \mathbf S^1$ and $\bar g=\hat g+\hat u^2\mathrm ds^2$. Then Proposition \ref{Prop: vanish normal ricci limit} tells us that $(\bar\Sigma,\bar g)$ has nonnegative Ricci curvature and $\Ric_{\bar g}(\partial_s,\partial_s)\equiv 0$. Through investigating the variation induced by the isometric $\mathbf S^1$-action of $(\bar\Sigma,\bar g)$ we obtain $\hat\Delta\hat u=0$, where $\hat\Delta$ denotes the Laplace-Beltrami operator of $\hat\Sigma$. The same argument as in the proof for Theorem \ref{Thm: main1} implies that $\hat u$ is a positive constant. As a result, $(\hat\Sigma,\hat g)$ has nonnegative Ricci curvature and so the finite area of $\hat\Sigma$ yields the compactness of $\hat\Sigma$. From the connectedness of $\hat\Sigma_k$ and the local graphical convergence, we conclude that $\hat\Sigma_k$ is a graph over $\hat\Sigma$ when $k$ is sufficiently large. Hence $\hat\Sigma$ admits a smooth map to $\mathbf S^2\times T^{n-3}$ with nonzero degree and $\sys_2'(\hat\Sigma,\hat g,f_1|_{\hat\Sigma})=4\pi$. Since $\hat\Sigma$ is an area-minimizing boundary in $(M,g)$ and $R(g)\geq 2$, the standard foliation argument as in \cite[Proposition 3.4]{Z2020} yields that $M$ is isometrically covered by the Riemannian product $\mathbf S^2\times \mathbf R^{n-2}$.
\end{proof}
\begin{remark}
We point out that there is an alternative way to prove the rigidity part of Theorem \ref{Thm: main3} with Cheeger-Gromoll splitting theorem. In fact, one can show that $(M,g)$ has nonnegative Ricci curvature based on the idea in Proposition \ref{Prop: scalar 2 limit} and the use of $h_\epsilon$-minimizing boundaries.
\end{remark}

%\section{Alternative proof for Theorem \ref{Thm: main2} and Theorem \ref{Thm: main3}}

%In this section, we devote to prove ???

\bibliography{bib}
\bibliographystyle{amsplain}
\end{document}